\documentclass[a4paper,11pt,leqno]{amsart}

\usepackage{amsmath}
\usepackage{amsthm}
\usepackage{amsfonts}
\usepackage{amssymb}

\usepackage{enumerate}
\usepackage{colonequals}

\newcommand{\inv}{^{-1}}
\newcommand{\eps}{\varepsilon}
\newcommand{\N}{\mathbb N}
\newcommand{\R}{\mathbb R}

\newtheorem{theorem}{Theorem}[section]
\newtheorem*{theorem*}{Theorem}{\bf}{\it}
\newtheorem{proposition}[theorem]{Proposition}
\newtheorem*{proposition*}{Proposition}{\bf}{\it}
\newtheorem{lemma}[theorem]{Lemma}
\newtheorem*{lemma*}{Lemma}{\bf}{\it}
\newtheorem{corollary}[theorem]{Corollary}
\theoremstyle{definition}
\newtheorem{definition}[theorem]{Definition}
\newtheorem*{definition*}{Definition}
\theoremstyle{remark}
\newtheorem{remark}[theorem]{Remark}
\numberwithin{equation}{section}

\begin{document}

\title[A characterization of BLD-mappings]{A characterization of BLD-mappings between metric spaces}
\author{Rami Luisto}
\address{Department of Mathematics and 
  Statistics, P.O. Box 68 (Gustaf H\"allstr\"omin katu 2b), 
  FI-00014 University of Helsinki, Finland}
\email{rami.luisto@iki.fi}
\thanks{The author is supported by
 the V\"ais\"al\"a foundation.}

%\date{\today}

\maketitle

\begin{abstract}
  We prove a characterization theorem for 
  BLD-mappings between complete locally compact
  path-metric spaces.
  As a corollary we obtain a sharp
  limit theorem for BLD-mappings.
\end{abstract}

\section{Introduction}
\label{sec:Intro}

The class of BLD-mappings was introduced in
\cite{MartioVaisala} as mappings that 
preserve solutions of certain elliptic partial differential 
equations. In that paper Martio and V\"ais\"al\"a 
showed, among other results, that the class
of BLD-mappings has several equivalent definitions. In this
paper we use the following geometric definition. 
For the definitions of 
the length of a path,
path-metric spaces
and branched covers, see Section \ref{sec:Preli}.
\begin{definition*}
  Given $L \geq 1$, a branched cover $ f \colon X \to Y$ between metric spaces
  is a mapping of \emph{Bounded Length Distortion},
  or ($L$-)BLD for short, if
  for all paths $\gamma \colon [0,1] \to X$, we have
  \begin{align*}
    L \inv \ell(\gamma)
    \leq \ell(f \circ \gamma)
    \leq L \ell(\gamma).
    \tag{BLD}
  \end{align*}
\end{definition*}

In \cite{MartioVaisala} Martio and V\"ais\"al\"a defined BLD-mappings
as a subclass of the so called \emph{quasiregular mappings}, see e.g.\
\cite{MartioVaisala} or \cite{Rickman-book},
and showed that this analytic definition
is equivalent to the geometric definition given above. There are, 
however, also other characterizations of BLD-mappings 
between Euclidean spaces in \cite{MartioVaisala}
although they are not explicitly stated as such.
In this paper we state these characterizing properties 
and prove a characterization (Theorem \ref{thm:BLDDefinitions}),
which shows that the equivalent definitions for BLD-mappings
in \cite{MartioVaisala} hold true also in the setting of complete
locally compact path-metric spaces.

Note that the path-length condition (BLD)
is required for \emph{all} paths. This requirement
implies that the lifts of rectifiable paths are also
rectifiable. (For terminology, see Section \ref{sec:Preli}.)
Requiring the path-length condition (BLD)
only for rectifiable paths 
gives rise to the class of \emph{weak BLD-mappings},
that have been studied e.g.\ by Haj{\l}asz and Malekzadeh, 
see \cite{MalekzadehHajlasz2}.
In Euclidean spaces the two definitions are equivalent
(see \cite{MalekzadehHajlasz})
but, in general, weak BLD-mappings form a strictly larger class of mappings.
For example the identity map $\mathbb{H}_1 \to \R^3$ from the first
Heisenberg group to the Euclidean $3$-space is a weak BLD-mapping but
not a BLD-mapping.

A mapping $f \colon X \to Y$ between metric
spaces is 
\emph{$L$-LQ (Lipschitz Quotient)} if, 
for all $x \in X$ and $r>0$,
\begin{align*}
  B_Y(f(x),L \inv r)
  \subset f \left(B_X(x,r)\right)
  \subset B_Y(f(x),L r).
  \tag{LQ}
\end{align*}
LQ-mappings were introduced with this name in
\cite{BJLPS}. Note, however, that
Martio and V\"ais\"al\"a show already in \cite[Lemma 4.6.]{MartioVaisala}
that BLD-mappings satisfy the (LQ) property locally in Euclidean domains.
For mappings between complete and locally
compact path-metric spaces the
definition of $L$-LQ-mappings is equivalent to
a local one; see Lemma \ref{lemma:LQlocal}.
The definition of LQ-mappings immediately yields that 
LQ-mappings are open, but they are not necessarily
discrete; the projection map $\R^2 \to \R$ is a trivial example,
but see also \cite{Csornyei} for
a construction of Cs\"ornyei for an LQ-mapping $f \colon \R^3 \to \R^2$
with a point $P \in \R^2$ such that $f \inv (\{ P \})$ contains a plane.

Let $f \colon X \to Y$ be a continuous mapping between 
path-connected metric spaces.
As in \cite[II.4]{Rickman-book} and
\cite[p. 491]{HeinonenRickman} we set
\begin{align*}
  L(x,f,r) 
  &\colonequals \sup \{ d(f(x),f(y)) \mid y \in \partial B(x,r) \}, \\
  l(x,f,r) 
  &\colonequals \inf \{ d(f(x),f(y)) \mid y \in \partial B(x,r) \}, \\
  L^\ast(x,f,r) 
  &\colonequals \sup \{ d(x,y) \mid y \in \partial U(x,f,r) \},
  \intertext{and}
  l^\ast(x,f,r) 
  &\colonequals \inf \{ d(x,y) \mid y \in \partial U(x,f,r) \},
\end{align*}
where $U(x,f,r)$ the component of $f \inv \left(B(f(x),r)\right)$ containing
$x$. Note that since $X$ is a path-metric space, for every $x \in X$
$\partial U(x,f,r) \neq \emptyset$ for all $r>0$ small enough
when $f$ is not a constant map.

A mapping $f \colon X \to Y$ is \emph{$L$-radial}
if for all $x \in X$ there exists a radius $r_0 > 0$ such that
for all $r < r_0$
\begin{align*}\tag{R}
  L(x,f,r) 
  \leq Lr 
  \quad
  \text{ and }
  \quad
  l(x,f,r) 
  \geq L \inv r.
\end{align*}
An equivalent definition
for radial mappings is given in Lemma \ref{lemma:RadialityDef2}.

Likewise we say that a branched cover $f \colon X \to Y$ is \emph{$L$-coradial}
if for all points $x \in X$ there exists a radius $r_0 > 0$
such that for all $r < r_0$
\begin{align*}\tag{R*}
  L^\ast(x,f,r) 
  \leq Lr 
  \quad
  \text{ and }
  \quad
  l^\ast(x,f,r) 
  \geq L \inv r.
\end{align*}
As branched covers 
coradial mappings are continuous, open and discrete by definition.
For radial mappings the radiality condition (R) immediately implies 
that an $L$-radial mapping is both discrete and locally $L$-Lipschitz.
On the other hand radial maps are not
necessarily open, as the example $(x,y) \mapsto (|x|,y)$
defined in the Euclidean plane shows. 

Our first main theorem is the following characterization.
\begin{theorem}\label{thm:BLDDefinitions}
  Let $f \colon X \to Y$ be a continuous mapping
  between two complete locally compact
  path-metric spaces and $L \geq 1$. Then the following are
  equivalent:
  \begin{enumerate}[(i)]
  \item $f$ is an $L$-BLD-mapping,
  \item $f$ is a discrete $L$-LQ-mapping,
  \item $f$ is an open $L$-radial mapping, and
  \item $f$ is an $L$-coradial mapping.
  \end{enumerate}
\end{theorem}
Theorem \ref{thm:BLDDefinitions} generalizes and extends an earlier result by
the author, see \cite[Theorem 1.1]{Luisto-NoteOnLocalToGlobal}.
As mentioned, this result
is known in the Euclidean setting \cite{MartioVaisala}. The implication
$(i) \Rightarrow (ii)$ is observed in the setting of generalized manifolds of type $A$
in \cite{HeinonenRickman}, and more recently the implication 
$(i) \Rightarrow (iii)$ is noted in a setting similar to \cite{HeinonenRickman} 
under additional assumptions on spaces $X$ and $Y$
by Guo and Williams \cite{GuoWilliams}.
The implication $(iii) \Rightarrow (ii)$ is implicitly due to
Lytchak in a purely metric setting without notions of branched covers, see
\cite[Section 3.1 and Proposition 4.3]{LytchakOpen}.
Furthermore in \cite[Theorem 4.5]{HeinonenRickman} it is shown that 
a mapping between 
quasiconvex generalized manifolds is BLD if 
and only if it is locally regular in the sense of David 
and Semmes, see \cite[Definition 12.1]{DavidSemmes}. 
That equivalence does not, however, preserve the
constant $L$.

Locally uniform limits of $L$-LQ mappings are $L$-LQ in a 
very general setting and so Theorem \ref{thm:BLDDefinitions}
yields that the $L$-BLD condition passes to limits of BLD-mapping
packages (defined in Section \ref{sec:Limits})
when the limiting map is discrete. More precisely, we have the following
theorem.
\begin{theorem}\label{thm:BLDLimit}
  Let $(X_j,x_j)$ and $(Y_j,y_j)$ be two pointed sequences
  of locally compact and complete path-metric spaces.
  Suppose the sequence of pointed mapping packages
  $\left( (X_j,x_j), (Y_j,y_j), f_j \right)$,
  where each 
  $f_j \colon (X_j,x_j) \to (Y_j,y_j)$ is $L$-BLD,
  converges to a mapping package
  $((X,x_0), (Y,y_0), f)$
  where $f$ is discrete.
  Then $f$ is $L$-BLD.
\end{theorem}
As an immediate corollary we get a result for fixed spaces.
\begin{corollary}
  Let $X$ and $Y$ be locally compact complete path-metric spaces
  and suppose $(f_j)$ is a sequence of $L$-BLD-mappings
  $X \to Y$ converging pointwise to a
  continuous discrete mapping $f \colon X \to Y$.
  Then $f$ is $L$-BLD.
\end{corollary}

In the setting of path-metric generalized manifolds of type $A$, 
we may deduce
the discreteness of a limit of $L$-BLD mappings 
from uniform bounds for local multiplicity. Thus
in this setting we have the following stronger result.
\begin{theorem}\label{thm:BLDLimit2}
  Let $(M_j,x_j)$ and $(N_j,y_j)$ be two pointed sequences
  of path-metric generalized $n$-manifolds of type $A$ 
  with uniform constants.
  Suppose the sequence of pointed mapping packages
  $ \left( (M_j,x_j), (N_j,y_j), f_j \right)$
  converges to a mapping package
  $((M,x_0), (N,y_0), f)$, 
  where each $f_j \colon M_j \to N_j$ is
  $L$-BLD. Then $f$ is $L$-BLD.
\end{theorem}

We again have a corresponding result for
fixed spaces as a corollary.
\begin{corollary}
  Let $M$ and $N$ be generalized $n$-manifolds
  of type $A$ and suppose
  $(f_j)$ is a sequence of $L$-BLD-mappings
  $M \to N$ converging pointwise to a
  continuous mapping $f \colon M \to N$.
  Then $f$ is $L$-BLD.
\end{corollary}
Note that in \cite{MartioVaisala} Martio and V\"ais\"al\"a
prove a corresponding sharp
result in the Euclidean setting, and in 
\cite{HeinonenKeith}
Heinonen and Keith show that in the setting
of the so called generalized manifolds of type A
the limit map is $K$-BLD with $K$ depending only on the data.

\begin{remark}
  The radiality conditions are also connected to quasiregular
  mappings. For example in \cite[Definition 4.1]{OnninenRajala} 
  quasiregular mappings are defined as orientation preserving branched
  covers for which the function 
  $H_f(x) \colonequals \limsup_{r \to 0} \frac{L(x,f,r)}{l(x,f,r)}$
  is everywhere finite
  and has bounded essential supremum;
  see also \cite{GuoWilliams}.
\end{remark}

\bigskip
\noindent
\textbf{Acknowledgements.} The author would like to thank,
once again, his advisor Pekka Pankka for introducing him to the world
of BLD geometry. The ideas presented in this manuscript
have been incubating for a long time and have been worked on during
several wonderful mathematical events, but the author would
like to mention the 
\emph{School in Geometric Analysis, Geometric analysis on 
Riemannian and singular metric spaces} at Lake Como School of Advanced Studies
in the fall of 2013 and the 
\emph{Research Term on Analysis and Geometry in Metric Spaces} at 
ICMAT in the spring of 2015 as especially fruitful events 
concerning the current work. The author is especially grateful for all the
interesting discussions with fellow mathematicians at these, and other,
events.

The contents of this paper were improved further 
by the author's 
discussions with Piotr Haj{\l}asz and his students while
visiting University of Pittsburgh in 2016. 
Finally, the thoroughness and comments of the anonymous referee are
gratefully acknowledged and have improved
the readability of the manuscript.

\section{Preliminary notions}
\label{sec:Preli}

A mapping between topological spaces
is said to be \emph{open} if the image
of every open set is open and \emph{discrete}
if the point inverses are discrete sets.
A continuous, discrete and open mapping
is called a \emph{branched cover}.

The \emph{length} $\ell(\beta)$ of a path 
$\beta \colon [0,1] \to X$ in a metric space is defined as
\begin{align*}
  \ell(\beta)
  \colonequals 
  \left\{
    \sum_{j=1}^{k} d(\beta(t_{j-1}),\beta(t_j)) \mid 
    0 = t_0 \leq \ldots \leq t_k = 1 
  \right\}.
\end{align*}
Paths with finite and infinite length are called
rectifiable and unrectifiable, respectively.
A metric space $(X,d)$ is 
a \emph{path-metric space} if
\begin{align*}
  d(x,y)
  = \inf 
  \left\{ 
    \ell(\gamma) \mid \gamma \colon [0,1] \to X, \gamma(0) = x , 
    \gamma(1) = y 
  \right\}
\end{align*}
for all $x,y \in X$.
In a similar vein, a metric space $(X,d)$ is ($C$-)quasiconvex
if for all $x,y \in X$ there exist a path $\beta \colon [0,1] \to X$
with $\beta(0)=x$, $\beta(1)=y$ and $\ell(\beta) \leq C d(x,y)$.
A $1$-quasiconvex space is called a \emph{geodesic space}.
Note that by the Hopf-Rinow theorem 
(see e.g.\ \cite[Hopf-Rinow theorem, p.9]{Gromov})
complete and locally compact path-metric spaces are geodesic and proper; i.e.\ 
closed balls are compact. Geodesic spaces are
always locally and globally connected.
Throughout this section $X$ and $Y$ are
locally compact and complete path-metric spaces and
$f \colon X \to Y$ a branched cover.
Furthermore we denote
\begin{align*}
  N(y_0,f,A) 
  \colonequals \# \left( A \cap f \inv \{ y_0 \} \right)
  \quad
  \text{ and }
  \quad
  N(f,A) 
  \colonequals \sup_{y \in A} N(y,f,A),
\end{align*}
where $A \subset X$ and $y_0 \in Y$.

We follow the conventions of \cite{Rickman-book} and
say that $U \subset X$ is a \emph{normal domain} (for
$f$) if $U$ is a precompact domain such that 
$\partial f (U) = f (\partial U)$. A normal domain $U$
is \emph{a normal neighbourhood} of $x \in U$ if
$\overline{U} \cap f \inv (\{ f(x) \}) = \{ x \}$.
By $U(x,f,r)$ we denote the component
of the open set $f \inv (B_Y(f(x),r))$ containing $x$.
The existence of arbitrarily small normal neighbourhoods 
is essential for the theory of branched covers. Heuristically
normal domains for branched covers have the same role as
completeness has for BLD-mappings. The following lemma
guarantees the existence of normal domains and the proof is 
the same as in \cite[Lemma I.4.9, p.19]{Rickman-book},
see also \cite[Lemma 5.1.]{Vaisala}.
\begin{lemma}\label{lemma:TopologicalNormalDomainLemma}
  Let $X$ and $Y$ be locally compact complete path-metric
  spaces and $f \colon X \to Y$ a branched cover.
  Then for every point $x \in X$ 
  there exists a radius $r_x > 0$ such that
  $U(x,f,r)$
  is a normal neighbourhood of $x$ for
  any $r \in (0,r_x)$.
\end{lemma}

The following corollary is an immediate consequence
of
Lemma \ref{lemma:TopologicalNormalDomainLemma}
and the precompactness of normal domains.
\begin{corollary}\label{coro:VaisalaTrick}
  Let $f \colon X \to Y$ be a branched cover between
  locally compact complete path-metric spaces
  and $U \subset X$ a normal domain.
  Then for any $y \in f(U)$ there exists a radius $r_y > 0$
  such that for every $r \in (0, r_y)$ the domains
  $U(z,f,r)$ are disjoint normal 
  neighbourhoods of the points $z \in U \cap f \inv ( \{ y \})$
  with
  \begin{align*}
    U \cap f \inv \left( B(y,r) \right)
    = \bigcup_{z \in U \cap f \inv ( \{ y \})} U(z,f,r).
  \end{align*}
\end{corollary}

As noted in the introduction, the definition
of Lipschitz quotient mappings is equivalent to 
a local definition in the setting of
complete and locally compact path-metric spaces.
\begin{lemma}\label{lemma:LQlocal}
  Let $X$ and $Y$ be complete and locally compact
  path-metric spaces.
  Suppose $f \colon X \to Y$ is
  \emph{locally $L$-LQ}, i.e.\
  for all $x \in X$ there exist $r_0 > 0$ such that
  (LQ) holds for all $0 < r < r_0$.
  Then $f$ is $L$-LQ.
\end{lemma}
\begin{proof}
  Fix a point $x_0 \in X$ and denote by 
  $I$ the set of those $s \in (0,\infty)$
  for which (LQ) holds at $x_0$ for all $r \leq s$.
  By our assumption $I$ is a non-empty interval.
  Suppose $I \neq (0,\infty)$. Then
  the supremum $\sup I$ exists and
  a straightforward calculation shows that 
  $\sup I \in I$. 
  Since $X$ is a proper geodesic space, the set 
  $\partial B_X(x_0, \sup I)$ is a non-empty
  compact set. By applying the local (LQ)
  condition at all points of the boundary
  $\partial B_X(x_0, \sup I)$ 
  we see by the compactness of the boundary
  that  there exists $\eps > 0$ with $\sup I + \eps \in I$,
  which is a contradiction.
  Thus $I = (0,\infty)$ and the claim holds true.
\end{proof}

As mentioned in the introduction, radial mappings have
an equivalent definition which we give next.
This is in fact the formulation used by Martio and V\"ais\"al\"a, see
\cite[Corollary 2.13]{MartioVaisala}. 
\begin{lemma}\label{lemma:RadialityDef2}
  A mapping $f \colon X \to Y$ between metric
  spaces is $L$-radial 
  if and only if for any point $x \in X$
  there exists a radius $r_0 > 0$ such that
  \begin{align*}\tag{R$^\#$}
    L \inv d(x,y)
    \leq d(f(x),f(y))
    \leq L d(x,y)
  \end{align*}
  for all $y \in B_X(x,r_0)$.
\end{lemma}
\begin{proof}
  The claim follows immediately as
  we note that the
  radius $r_0$ is the
  same as the radius in the definition
  of radial mappings.
\end{proof}

\subsection{Path-lifting methods}
The main tool in the proof of Theorem \ref{thm:BLDDefinitions}
is \emph{path-lifting}. Given a mapping
$f \colon X \to Y$ and a path $\beta \colon [0,1] \to Y$
we say that a path $\tilde \beta \colon I \to X$, where $I$
is an interval containing $0$, is a \emph{lift} of 
$\beta$ if $f \circ \tilde \beta = \beta|_I$.
A lift is called a \emph{maximal lift} if it is not
a proper restriction of another lift. Finally a lift
is a \emph{total lift} if $I = [0,1]$.
The existence of lifts give rise to maximal lifts
via a straightforward
Zorn's lemma argument, see e.g.\ 
\cite[Theorem 3.2.\ p.22]{Rickman-book}.

The following path-lifting theorem of Floyd \cite{Floyd}
would be sufficient for the purposes of this paper.
Recall that a mapping is \emph{light} if the pre-image
of a point is totally disconnected; discrete mappings are
always light.
\begin{theorem*}[Floyd]\label{thm:Floyd}
  Suppose $f \colon X \to Y$ is a light mapping
  between two compact metric spaces. Then
  $f$ is open if and only if there exists for
  any path $\beta \colon [0,1] \to Y$ and any
  point $x \in f \inv (\beta(0))$ a total lift of 
  $\beta$ starting from $x$.
\end{theorem*}

We give, for the reader's convenience,
a self-contained proof of a special case
(Proposition \ref{Proposition:Luisto})
of Floyd's theorem. 
The idea in the proof of 
Proposition \ref{Proposition:Luisto} 
is partly motivated by the proof of the
classical application of the Baire category theorem
e.g.\ in \cite[2.5.56, p.76]{KaczorNowak}.
The core of the proof is a Baire category theorem argument,
which we formulate as Lemma \ref{Lemma:BaireLemma} for the
sake of clarity. 

Let $f \colon X \to Y$ be a branched cover
between locally compact and complete path-metric spaces,
$U_0$ a normal domain for $f$ and $\beta \colon [0,1] \to f(U_0)$
a path. For any compact set $J \subset [0,1]$ we say that 
an open interval $W_J \subset [0,1]$ intersecting $J$
is \emph{a lifting interval for $J$}
if there exists a point $t_J \in W_J \cap J$, a positive 
number $k_J \in \N$, and a radius
$r_J > 0$ such that
\begin{enumerate}[{(LI}1{)}]
\item $\# (U_0 \cap f \inv ( \beta(t)) ) = k_J$ for all $t \in W_J \cap J$,
\item 
  for $\{ z_J^1, \ldots , z_J^{k_J} \} \colonequals U_0 \cap f \inv ( \beta(t_J) )$
  we have
  \begin{align*}
    U_0 \cap f \inv \left(B_Y(\beta(t_J),r_J )\right)
    = \bigcup_{i=1}^{k_J} U(z_J^i,f,r_J),
  \end{align*}
  where the union on the right hand side is disjoint, and
\item there exists mappings
  $g_J^i\colon W_J \cap J \to U(z_J^i,f,r_J)$,
  for $i = 1, \ldots, k_J$,
  for which $f \circ g_J^i = \beta|_{W_J \cap J}$ and
  the images of the mappings $g_J^i$ 
  cover all of $U_0 \cap f \inv (\beta(W_J \cap J))$.
\end{enumerate}

\begin{lemma}\label{Lemma:BaireLemma}
  Let $f \colon X \to Y$ be a branched cover
  between locally compact and complete path-metric spaces,
  $U_0$ a normal domain for $f$, and $\beta \colon [0,1] \to f(U_0)$
  a path. For any compact set $J \subset [0,1]$, there exists
  a lifting interval $W_J$ of $J$.
\end{lemma}
\begin{proof}
  For any compact subset $J \subset [0,1]$ and $k \in \N$, we denote
  \begin{align*}
    N_J^{\geq k}
    \colonequals & \{ t \in J \mid \# (U_0 \cap f \inv (\{\beta(t)\})) \geq k \}, \\
    N_J^{\leq k}
    \colonequals & \{ t \in J \mid \# (U_0 \cap f \inv (\{\beta(t)\})) \leq k \},
  \end{align*}
  and $N_J^{k} = N_J^{\geq k} \cap N_J^{\leq k}$.

  Since $f$ is an open continuous map, the set $N_J^{\geq k}$ is open.
  Thus the complementary set
  $N_J^{\leq k} = J \setminus N_J^{\geq k+1}$
  is closed for all $k \in \N$. Since $f$ is discrete
  the sets $\{ N_J^{\leq k} \mid k \in \N \}$ 
  form a countable closed cover of the compact set $J$. 
  Thus, by the Baire category theorem
  there exists a minimal index $k_J \in \N$ for which the set
  $N_J^{\leq k_J}$ has interior points in $J$. Since $k_J$ is minimal,
  also the set $N_J^{k_J}$ has interior points in $J$. This
  means that there exists an open interval $V \subset [0,1]$
  with $V \cap J \subset N_J^{k_J}$,
  so for all $t \in V \cap J$
  \begin{align*}
    \# 
    \left(
      U \cap f \inv (\{ \beta(t) \}) 
    \right) 
    = k_J.
  \end{align*}
  Let $t_J \in V \cap J$ and denote
  $\{ z_J^1, \ldots, z_J^{k_J}\} \colonequals U_0 \cap f\inv (\{ \beta(t_J) \})$.
  Let also $r_J>0$ be a radius so small that the sets $U(z_J^i,f,r_J)$
  are disjoint normal neighbourhoods of the points $z_i$ for $i=1,\ldots, k_J$
  satisfying
  \begin{align*}
    U_0 \cap f \inv \left( B(\beta(t_J),r_J) \right)
    = \bigcup_{i=1}^{k_J} U(z_J^i,f,r_J)
  \end{align*}
  as in Corollary \ref{coro:VaisalaTrick}.
  Set $W_J \subset V$ to be an open interval around $t_J$
  with $\beta(W_J) \subset B(\beta(t_J),r_J)$.
  Since $f$ is an open map, the restriction of $f$ to the pre-image
  of $\beta(W_J \cap J)$ in $U_0$ is locally injective by the definition of $k_J$. 
  Thus for any compact set $K \subset W_J \cap J$
  the pre-image
  $f \inv (\beta (K))$ is compact and as a locally injective
  map between compact sets in Hausdorff spaces 
  the restriction of $f$
  to $f \inv ( \beta (K) )$ is a local homeomorphism. 
  This local inverse 
  yields maps
  \begin{align*}
    g_J^i\colon W_J \cap J \to U(z_J^i,f,r)
  \end{align*}
  satisfying $f \circ g_J^i = \beta|_{W \cap J}$,
  for $i = 1, \ldots, k_J$. Furthermore the images of these
  lifts cover all of $U_0 \cap f \inv (\beta(W_J \cap J))$.
\end{proof}

\begin{proposition}\label{Proposition:Luisto}
  Let $f \colon X \to Y$ be a branched cover between 
  locally compact and complete path-metric spaces. 
  Suppose $U_0$ is a normal domain for $f$ and let $\beta \colon [0,1] \to f(U_0)$
  be a path. Then for any $x_0 \in U_0 \cap f \inv (\{\beta(0)\})$ there
  exists a total lift of  $\beta$ starting from $x$, i.e.\
  a path $\tilde \beta \colon [0,1] \to U_0$ for which
  $\tilde\beta(0) = x_0$ and $f \circ \tilde\beta = \beta$.
\end{proposition}
\begin{proof}
  To use Lemma \ref{Lemma:BaireLemma}
  to construct lifts of $\beta$,
  let $\mathcal{I}$ be the collection of 
  all intervals $(a,b) \subset [0,1]$ such that for any $c \in (a,b)$
  and any $x \in U_0 \cap f \inv (\{ \beta(c) \})$ there exists a path
  $\alpha \colon (a,b) \to U_0$ with $\alpha(c) = x$ and 
  $f \circ \alpha = \beta|_{(a,b)}$. 
  The definition immediately yields
  that the collection $\mathcal{I}$ is closed under restrictions to
  open subintervals, finite non-empty intersections and finite unions when
  the union is an interval. Furthermore if 
  $(a_j,b_j) \in \mathcal{I}$, $j \in S$, is any collection
  with a connected union, a straightforward argument shows
  that also $(\inf_j a_j, \sup_j b_j) \in \mathcal{I}$.
  We conclude that $\mathcal{I}$ is closed under
  arbitrary connected unions. The rest of
  the proof is dedicated into showing first that
  $\mathcal{I}$ is non-empty, second that
  $\cup \mathcal{I}$ is dense in $[0,1]$ and finally
  that $(0,1) \in \mathcal{I}$.
  
  Applying Lemma \ref{Lemma:BaireLemma}
  with $J$ a closed subinterval of $[0,1]$
  yields an interval $W_J \subset [a,b]$ and
  $k_J$ lifts $g_J^1, \ldots, g_J^{k_J}$ of
  $\beta|_{W_J}$ covering all of the pre-image 
  of $\beta(W_J)$ in $U_0$.
  From this we conclude that
  $\mathcal{I}$ contains $ W_J$ and thus
  is not empty.
  In fact, this argument shows that
  every closed interval $J \subset [0,1]$ contains
  an open subinterval $I \subset W_J \cap \operatorname{int} J$
  with $I \in \mathcal{I}$.
  Thus $\cup \mathcal{I}$ is dense in $[0,1]$. 

  We show next that $(0,1) \in \mathcal{I}$.
  The collection $\mathcal{I}$ is closed under connected
  unions, so it suffices to show that $\cup \mathcal{I} = (0,1)$.
  Suppose $\cup \mathcal{I} \neq (0,1)$. 
  Since the collection $\mathcal{I}$ is closed under connected
  unions, we may write $\cup \mathcal{I}$ as 
  a countable union of disjoint open intervals $(a_j,b_j)$ for $j \in \N$.
  We apply Lemma \ref{Lemma:BaireLemma}
  to the compact set
  \begin{align*}
    J 
    \colonequals [0,1] \setminus \bigcup_j (a_j,b_j)
  \end{align*}
  and obtain a lifting interval $W_J$ for $J$ together with the related
  points $z_J^1, \ldots , z_J^{k_J}$ as in (LI2).
  We claim that $W_J \in \mathcal{I}$. 
  Let $c \in W_J$ and fix a pre-image $z_0 \in U_0 \cap f \inv ( \{ \beta(c) \} ) $.
  Let $i_0 \in \{ 1, \ldots, k_J \}$ 
  be the unique index for which $z_0 \in U(z_J^{i_0},f,r_J)$.
  To define a lift $\gamma \colon W_J \to U(z_J^{i_0},f,r_J) \subset U_0$
  of $\beta|_{W_J}$ we set first $\gamma|_{W_J \cap J} = g_J^{i_0}$.
  Since $\cup \mathcal{I} = \bigcup_{j \in \N} (a_j,b_j)$ is 
  dense in $[0,1]$, there exists intervals $(a_j,b_j)$ intersecting
  $W_J$.
  For each of the intervals $(a_j,b_j) \subset W_J$ there
  exists at least one lift $\alpha_j$ of $\beta|_{(a_j,b_j)}$
  with $|\alpha_j| \cap U(z_J^{i_0},f,r_J) \neq \emptyset$;
  if $c \in (a_j,b_j)$ we take the lift of $\beta|_{(a_j,b_j)}$
  containing $z_0$, otherwise we take any one
  of the finitely many possibilities.
  Since
  \begin{align*}
    \beta \left( ( a_j, b_j) \right)
    \subset \beta(W_J) 
    \subset B(\beta(t_J),r_J)
  \end{align*}
  and $U(z_J^{i_0},f,r_J)$
  is normal domain, $|\alpha_j| \subset U(z_J^{i_0},f,r_J)$.
  Thus we may set $\gamma|_{(a_j,b_j)} = \alpha_j$.
  For the two possible intervals intersecting $W_J$
  but not contained in $W_J$ we fix lifts in $U_0$ in a similar vein
  by studying the intersection of such intervervals with $W_J$.
  From the definition of 
  a lifting interval
  we immediately see that the lift
  $\gamma$ thus defined is continuous.
  Thus $W_J \in \mathcal{I}$,
  which is a contradiction with the definition of $J$
  and we conclude that $\cup \mathcal{I} = (0,1)$.
  Since $\mathcal{I}$ is closed under connected unions,
  $(0,1) \in \mathcal{I}$.

  The fact that $(0,1) \in \mathcal{I}$
  implies the existence of lifts of $\beta|_{(0,1)}$.
  To conclude the proof we need to
  show that there exists a lift of 
  the whole path $\beta$
  with $\tilde \beta(0) = x_0$.
  Let $r > 0$ be so small that
  $U(x_0,f,r)$ is a normal neighbourhood of 
  $x_0$. Take $c \in (0,1)$ such that
  $\beta|_{(0,c]} \subset B(\beta(0),r)$
  and fix $x \in U(x_0,f,r) \cap f \inv (\{ \beta(c) \})$.
  Since $(0,1) \in \mathcal{I}$, there exists
  a lift $\tilde \beta \colon (0,1) \to U_0$ of $\beta|_{(0,1)}$
  with $\tilde \beta(c) = x$. The restriction
  $\tilde \beta|_{(0,c)} \colon (0,c) \to U(x_0,f,r)$
  is a lift of $\beta \colon (0,c) \to B(\beta(0),r)$,
  and since $U(x_0,f,r)$ is a normal neighbourhood
  of $x_0$, $\lim_{s \to 0} \tilde \beta(s) = x_0$. Thus the lift
  $\tilde \beta \colon (0,1) \to U_0$ extends to a lift
  $\tilde \beta' \colon [0,1) \to U_0$ with $\tilde \beta(0) = x_0$.
  Since $U_0$ is a normal neighbourhood and $f$ is discrete,
  the limit $\lim_{s \to 1} \tilde \beta(s)$ will equal
  one of the finitely many pre-images 
  $U_0 \cap f \inv (\{\beta(1) \})$. Thus the
  lift extends to the whole interval $[0,1]$ and 
  the claim holds true.
\end{proof}

\begin{remark}
  We note that the proofs of 
  Lemma \ref{Lemma:BaireLemma} and Proposition \ref{Proposition:Luisto} 
  are valid for branched covers between
  locally compact, locally and globally path-connected Hausdorff spaces.

  In the proof of Lemma \ref{Lemma:BaireLemma} the fact that
  $\beta$ is a \emph{path}, i.e.\ that it is defined on an interval,
  has no role. Indeed the result holds true when
  ``lifting interval'' is replaced by ``lifting domain''
  for any mapping 
  $\beta \colon Z \to f(U_0)$
  where $Z$ is a Baire space.
  The proof of the Proposition \ref{Proposition:Luisto} does,
  on the contrary, rely heavily on the fact that $\beta$ is defined 
  on an interval.
\end{remark}

By the theorem of Floyd
and Proposition \ref{Proposition:Luisto} we do not 
obtain the maximal family of lifts as in Rickman's
path lifting theorem \cite{Rickman-PathLifting}.
A single total lift of a given path within 
a normal domain is, however, sufficient for our methods.
For rectifiable paths and BLD-mappings the local
path-lifting extends to a global lift in the following sense.
\begin{lemma}\label{lemma:BLDPathLifting}
   Let $f \colon X \to Y$ be an $L$-BLD mapping
   between two locally compact and complete path-metric spaces
   for $L \geq 1$. Then
   any rectifiable path $\beta \colon [0,1] \to Y$
   has a total lift starting from any point
   $x_0 \in f \inv (\{ \beta(0) \})$.
\end{lemma}
\begin{proof}
  The claim follows immediately from noting that lifts of rectifiable
  paths are contained in closed balls which are compact in the setting
  of the lemma.
\end{proof}

Proposition \ref{Proposition:Luisto} also yields
the following corollary. 
\begin{corollary}\label{coro:ConnectedPreImages}
  Let $f \colon X \to Y$ be a branched cover between 
  locally compact complete path-metric spaces.
  Let $x \in X$ and let $U$ be a normal neighbourhood of $x$.
  Then for any connected open set $W \subset f(U)$ with $f(x) \in W$
  the pre-image $U \cap f \inv W$ is connected.
\end{corollary}
\begin{proof}
  Let $y \in W$ and connect $f(y)$ to $f(x)$ with a path $\alpha \colon [0,1] \to W$.
  For any point $z \in f \inv \{ y \}$ there exists, by 
  Proposition \ref{Proposition:Luisto}, a lift $\tilde \alpha \colon [0,1] \to U$
  with $\alpha(0) = z$. Thus $z$ belongs to the same component of $f \inv W$ as
  $x$, which proves the claim.
\end{proof}

\section{Proof of Theorem \ref{thm:BLDDefinitions} }\label{sec:BLD}
\label{sec:Main}

In the proof of Theorem \ref{thm:BLDDefinitions} 
the most involved part is to show that
the radiality condition (R) implies the path-length condition
(BLD) with the same constant (see also \cite{MalekzadehHajlasz}).
This implication is true in the setting of general path-metric spaces
and does not require $f$ to be open.
Thus we state it as the following separate lemma.
We thank Jussi V\"ais\"al\"a for this short and elementary
proof.
\begin{lemma}\label{lemma:RadialIsBLD}
  Let $f \colon X \to Y$ be an $L$-radial mapping between path-metric spaces
  and $\beta \colon [0,1] \to X$ a path. Then
  \begin{align*}
    L \inv \ell(\beta)
    \leq \ell(f \circ \beta)
    \leq L \ell(\beta).
  \end{align*}
\end{lemma}
\begin{proof}
  Let $\beta \colon [0,1] \to X$ be a path.
  The inequality 
  $ \ell(f \circ \beta) \leq L \ell(\beta)$
  immediately follows from the fact that
  between path-metric spaces an $L$-radial mapping
  is an $L$-Lipschitz mapping.

  To prove the inequality
  $L \inv \ell(\beta) \leq \ell(f \circ \beta)$,
  let $K = \{ x_0, \ldots , x_k \} \subset [0,1]$ with
  \begin{align*}
    0 
    = x_0 
    < x_1
    < \cdots
    < x_k 
    = 1
  \end{align*}
  and denote
  \begin{align*}
    C 
    \colonequals \sum_{i=1}^k d(\beta(x_{i-1}),\beta(x_i)).
  \end{align*}
  Also let $A$ be the set of those points $a \in [0,1]$
  for which there exists a finite set 
  $J(a) \colonequals \{ t_0, \ldots , t_m \} $ such that
  \begin{enumerate}
  \item $0 = t_0 < \cdots < t_m = a$,
  \item $K \cap [0,a] \subset J(a)$, and
  \item 
    $\sum_{i=1}^m d(\beta(t_{i-1}),\beta(t_i)) \leq L\sum_{i=1}^k d(f(\beta(t_{i-1})),f(\beta(t_i)))$.
  \end{enumerate}
  We note that $0 \in A$, so $A$ is not empty. Thus the supremum of $A$
  exists and by applying the radiality 
  condition at the point $f(\beta(\sup A))$ we observe first that 
  $\sup A \in A$ and then that $\sup A = 1$.
  Now for $J(1) \equalscolon \{ t_0, \ldots , t_N \}$
  we have by the triangle inequality that
  \begin{align*}
    C
    &= \sum_{i=1}^k d(\beta(x_{i-1}),\beta(x_i)) 
    \leq \sum_{i=1}^N d(\beta(t_{i-1}),\beta(t_i)) \\ 
    &\leq L\sum_{i=1}^N d(f(\beta(t_{i-1})),f(\beta(t_i))) 
    \leq L \ell(f \circ \beta).
  \end{align*}
  Thus we have that $\ell(\beta) \leq L \ell(f \circ \beta)$.
\end{proof}

\begin{proof}[Proof of Theorem \ref{thm:BLDDefinitions}]
  We note that under any of the conditions $(i)$--$(iv)$,
  the mapping $f$ is an $L$-Lipschitz branched cover.
  We prove the theorem by two sequences of implications, 
  showing first 
  $(i) \Rightarrow (ii) \Rightarrow (iii) \Rightarrow (i)$
  and then completing the equivalence by showing
  $(i) \Rightarrow (iv) \Rightarrow (iii)$  
  
  The proof of the implication $(i) \Rightarrow (ii)$
  is essentially from \cite[Proposition 4.13 and Remark 4.16(c)]{HeinonenRickman}.
  To show that an $L$-BLD-mapping is an $L$-LQ mapping
  it suffices to show that the co-Lipschitz condition
  $B_Y(f(x),L\inv r) \subset f(B_X(x,r))$ holds, since the
  other inclusion is equivalent to the mapping being 
  $L$-Lipschitz. Since
  $Y$ is a proper space, the bounded set 
  $f(B_X(x,r))$ is precompact. Thus we
  may fix $z_0 \in \partial f(B_X(x,r))$ with
  \begin{align*}
    d(f(x),z_0)
    = d\left(f(x), \partial f(B_X(x,r))\right),
  \end{align*}
  and a geodesic $\beta \colon [0,1] \to Y$
  with $\beta(0) = f(x)$ and $\beta(1) = z_0$.
  By Lemma \ref{lemma:BLDPathLifting}
  there exists a total lift $\alpha$ of this path starting
  from $x$.
  On the other hand, since $f$ is an open map and
  \begin{align*}
    \alpha(1) 
    \in f \inv (\{ z_0 \})
    \subset f \inv \left( \partial f(B_X(x,r)) \right),
  \end{align*}
  we have $\alpha(1) \in \partial B_X(x,r)$.
  Thus 
  \begin{align*}
    \ell(\alpha) 
    \geq d(\alpha(0),\alpha(1))
    \geq d(x,\partial B_X(x,r))
    = r.
  \end{align*}
  By combining this with (BLD)
  and the fact that $\beta$ is a geodesic
  \begin{align*}
    d\left(f(x), \partial f(B(x,r))\right)
    = d(f(x),z_0)
    = \ell(\beta) 
    \geq L \inv \ell(\alpha)
    \geq L \inv r.
  \end{align*}
  Thus
  $B_Y(f(x),L \inv r) \subset f(B_X(x,r))$.

  Suppose next that $f$ is a discrete $L$-LQ-mapping. 
  Let $x \in X$. Since $f$ is discrete, there is a positive
  distance $r_0 \colonequals d\left(x, f \inv (\{ f(x) \}) \setminus \{ x \}\right))$.
  We claim that for any $r < r_0/2$ we have
  \begin{align*}
    d\left(f(x), (f \partial B_X(x,r))\right)
    \geq L \inv r.
  \end{align*}
  To see this let $z \in \partial B_X(x,r)$ and note that 
  $B_X(z,r) \cap f \inv (\{ f(x) \}) = \emptyset$, 
  so $f(x) \notin f(B_X(z,r))$. Since
  $f (B_X(z,r))$ contains the ball $B_Y(f(z),L \inv r)$, 
  we have $d(f(z),f(x)) \geq L \inv r$.
  Thus $f$ is $L$-radial at $x$, since the condition
  $L(x,f,r) \leq Lr$ is equivalent to the $L$-Lipschitz 
  condition.

  By Lemma \ref{lemma:RadialIsBLD} an $L$-radial branched cover
  is $L$-BLD.
  Thus we have shown that the conditions $(i)$, $(ii)$ and $(iii)$
  are equivalent.

  The proof that an $L$-BLD mapping is $L$-coradial, let $x \in X$ and
  fix $r_x$ as in Lemma \ref{lemma:TopologicalNormalDomainLemma}.
  Denote
  \begin{align*}
    r_0 
    \colonequals
    \min \{ r_x, d\left(x, f \inv (\{ f(x) \}) \setminus \{ x \} \right) /(2L) \}
  \end{align*}
  and fix $r < r_0$.
  Let $z \in \partial B_Y(f(x),r)$ and let 
  $\beta_z \colon [0,1] \to \overline B_Y(f(x),r)$ be a geodesic with
  $\beta_z(0) = f(x)$ and $\beta_z(1) = z$.
  By Lemma \ref{lemma:BLDPathLifting} 
  for any $w \in U(x,f,r) \cap f \inv (\{ z \})$
  there exists a lift $\tilde \beta_z$ of $\beta_z$
  with $\tilde \beta_z (0) = x$ and $\tilde \beta_z (1) = w$.
  Since $f$ is $L$-BLD, $\ell(\tilde \beta_z) \leq Lr$.
  Thus for all $r < r_0$ we have $L^\ast(x,f,r) \leq Lr$.
  Since the condition $l^\ast(x,f,r) \geq L \inv r$ is equivalent
  to the $L$-Lipschitz condition, $f$ is $L$-coradial at $x$.

  Suppose finally that $f$ is a branched cover which satisfies the co-radiality
  condition (R*) with constant $L \geq 1$. Fix $x \in X$ and 
  let $r_0 > 0$ be such that for all $r < r_0$
  the normal neighbourhoods $U(x,f,r)$ satisfy
  (R*). By Corollary \ref{coro:ConnectedPreImages} we have 
  for all $r < r_0$ that 
  \begin{align*}
    U(x,f,r) 
    = U(x,f,r_0) \cap f \inv ( B_Y(f(x),r)).
  \end{align*}
  Then for each point $z \in U(x,f,r_0)$
  there exists a radius $r_z > 0$ such that
  $z \in \partial U(x,f,r_z)$.
  The condition (R*) then implies that 
  $d(z,x) \in [L \inv r , L r]$, so the mapping $f$ is $L$-radial.
\end{proof}

\section{Limit theorems for BLD-mappings }
\label{sec:Limits}

To show that the pointwise limit of 
$L$-LQ mappings $f_j \colon X \to Y$
between proper metric spaces
is an $L$-LQ mapping is a straightforward
calculation. In this section we
show that a similar limit result holds
in the setting of sequences of pointed
mapping packages. 

A \emph{pointed mapping package} is a triple
$\left((X,x_0),(Y,y_0),f\right)$ where $X$ and $Y$
are locally compact and complete 
path-metric spaces
having fixed base-points $x_0 \in X$, $y_0 \in Y$,
and $f \colon X \to Y$ is a continuous
mapping satisfying $f(x_0) = y_0$.
We define the convergence of a sequence of pointed
mapping packages as in 
\cite[Definition 8.18 and Lemma 8.19]{DavidSemmes},
see also
\cite[Definition 3.8]{KleinerMacKay} and
\cite[Definition 2.1]{GuyCDavid-Tangents}.
For $A \subset X$ we denote $N_\eps (A) \colonequals B_X(A, \eps)$.
A map $\phi \colon (X,x_0) \to (Y,y_0)$ between 
pointed metric spaces is called an 
\emph{$\eps$-quasi-isometry} if
\begin{enumerate}[(i)]
\item for all $a,b \in B_X(x_0, \eps \inv )$ we have
  $| d(\phi(a),\phi(b)) - d(a,b) | < \eps$, and
\item for all $\eps \leq r \leq \eps \inv $ we have 
  $N_\eps(\phi(B_X(x_0,r))) \supset B_Y(y_0, r- \eps)$.
\end{enumerate}

\begin{definition}\label{def:GH}
  \emph{Pointed mapping packages 
  $\left((X_j,x_j),(Y_j,y_j),f_j\right)$ 
  for $j \in \N$ converge to
  $\left((X,x_0),(Y,y_0),f\right)$} if the following conditions hold:
  \begin{enumerate}[(GH-i)]
  \item For every $r > 0$ and $i \in \N$
    there exists $\eps_i^{(r)} > 0$ and 
    $\eps_i^{(r)}$-quasi-isometries
    \begin{align*}
      & g_i^{(r)} \colon B_{X_i}(x_i,r) \to N_{\eps_i^{(r)}} \left( g_i^{(r)} \left(B_{X_i}(x_i,r)\right) \right) \subset X
      \quad \text{and} \\
      & h_i^{(r)} \colon B_{Y_i}(y_i,r) \to N_{\eps_i^{(r)}} \left( h_i^{(r)} \left(B_{Y_i}(y_i,r)\right) \right) \subset Y 
    \end{align*}
    so that $\eps_i^{(r)} \to 0$
    and for all $i \in \N$, $g_i^{(r)}(x_i) = x_0$, $h_i^{(r)}(y_i) = y_0$,
    \begin{align*}
      B(x_0,r - \eps_i^{(r)}) &\subset N_{\eps_i^{(r)}} \left( g_i^{(r)} \left(B_{X_i}(x_i,r) \right)\right) 
      \quad \text{and}\\
      B(y_0,r - \eps_i^{(r)}) &\subset N_{\eps_i^{(r)}} \left( h_i^{(r)} \left(B_{Y_i}(y_i,r) \right)\right).
    \end{align*}
    
  \item For any $x \in X$ and all $r > d(x,x_0)$
    we have
    $h_i^{(r)}(f_i(a_i))  \to f(a)$ as $i \to \infty$
    whenever $a_i \in X_i$ is a sequence of points with $g_i^{(r)}(a_i) \to a$
    as $i \to \infty$.
  \end{enumerate}
\end{definition}
Note that the condition (GH-i) in Definition \ref{def:GH}
is equivalent to saying that 
$(X_j, x_j) \to (X,x_0)$ and
$(Y_j, y_j) \to (Y,y_0)$ in the Gromov-Hausdorff sense, see
e.g.\ \cite{DavidSemmes}. For fixed spaces, the condition
(GH-ii) is just the pointwise convergence of mappings.

\begin{lemma}\label{lemma:LQLimit}
  Let $\left((X_j,x_j), (Y_j,y_j), f_j \right)$
  be a sequence of mapping packages converging to a 
  mapping package
  $\left((X,x_0), (Y,y_0), f \right)$.
  If all the mappings $f_j$ are $L$-LQ, then so is $f$.
\end{lemma}
\begin{proof}
  We show first that the limiting map $f$ is $L$-Lipschitz.
  Let $a,b \in X$ and fix
  a radius $R > 0$ with $a,b \in B_X(x_0,R)$
  and $f(a), f(b) \in B_Y(y_0,R)$.
  Fix two sequences of points 
  \begin{align*}
    a_i \in (g_i^{(R)})\inv \left(B_X(a,\eps_i^{(R)})\right) \subset X_i
    \quad
    \text{ and }
    \quad
    b_i \in (g_i^{(R)})\inv \left(B_Y(b,\eps_i^{(R)})\right) \subset X_i.
  \end{align*}
  By (GH-ii) we have
  $h_i^{(R)}(f_i(a_i)) \to f(a)$ and
  $h_i^{(R)}(f_i(b_i)) \to f(a)$, so
  since each $f_i$ is $L$-Lipschitz
  the triangle inequality yields
  \begin{align*}
    d_Y(f(a),f(b))
    \leq L d_{X}(a,b) + 2 \eps_i^{(R)}
  \end{align*}
  for all $i \in \N$. Thus $f$ is $L$-Lipschitz.

  To prove the claim it now suffices to show that
  $B_Y(f(x),r/L) \subset f(B_X(x,r))$ for each 
  $x \in X$ and $r > 0$. Let
  $z_0 \in B_Y(f(x),r/L)$. Fix a radius
  $r_0 < r$ such that $z_0 \in B_Y(f(x),r_0/L)$.
  Let also $R = 2L (d(x_0,x) + r)$ and let
  $(\eps_i^{(R)})$, $(g_i^{(R)})$ and $(h_i^{(R)})$ be as in
  Definition \ref{def:GH}.
  For each $i \in \N$ we take 
  \begin{align*}
    c_i 
    \in (g_i^{(R)})\inv \left(B_X(x,\eps_i^{(r)}) \right)
    \subset X_i.
  \end{align*}
  By (GH-ii) we have 
  $h_i(f_i(c_i)) \to f(x)$, so 
  $
  \delta_i 
  \colonequals d_Y(h_i(f_i(c_i)),f(x))
  \to 0
  $
  as $i \to \infty$.
  Likewise we fix for all $i \in \N$ a point
  \begin{align*}
    z_i 
    \in (h_i^{(R)})\inv \left(B_Y(z,\eps_i^{(R)})\right)
    \subset Y_i.
  \end{align*}
  Now the triangle inequality yields 
  $ d_{Y_i}(f(c_i),z_i) \leq r_0/L + (\delta_i + 2\eps_i^{(R)})$,
  since $h_i$ is an $\eps_i^{(R)}$-quasi-isometry.
  
  Since $f_i$ is $L$-LQ, we have
  \begin{align*}
    z_i
    \in B_{Y_i}(f_i(c_i), r_0 / L + (\delta_i +  2\eps_i^{(R)}))
    \subset f \left(B_{X_i}(c_i, r_0 +  L(\delta_i + 2 \eps_i^{(R)}))\right).
  \end{align*}
  Thus there exists, for each $i \in \N$, a point 
  \begin{align*}
    a_i 
    \in B_{X_i}(c_i, r_0 + L ( \delta_i + 2\eps_i^{(R)}))
    \cap f_i \inv (\{ z_i \}).
  \end{align*}

  Since $B_{X}(x_0,R_0)$ is precompact, we can pass to 
  subsequences $(\eps_j^{(R)})$, $(g_j^{(R)})$ and $(h_j^{(R)})$ 
  and assume that the sequence $(g_j^{(R)}(a_j))$ in $X$ 
  converges to a point
  in $\overline{B}_{X}(x_0,R_0)$. 
  Finally we note that since each $g_j^{(R)}$ 
  is an $\eps_j^{(R)}$-quasi-isometry, we have
  \begin{align*}
    g_j^{(R)}(a_j) 
    \in B_{X}(x,r_0 + L( \delta_j + 2\eps_j^{(R)}) 
    + 2\eps_j^{(R)}),
  \end{align*}
  so $a \in \overline{B}_X(x,r_0)$.
  Thus by (GH-ii), $f(a) = z$.
  This proves that we have
  \begin{align*}
    z
    \in f (\overline{B}_X(x,r_0))
    \subset f (B_X(x,r)).
  \end{align*}
  Since $z$ was arbitrary,
  $B_Y(f(x),L \inv r) \subset f (B_X(x,r))$
  and $f$ is $L$-LQ.
\end{proof}

Lemma \ref{lemma:LQLimit} yields immediately the proof of
Theorem \ref{thm:BLDLimit} when combined with our
characterization Theorem \ref{thm:BLDDefinitions}.
\begin{proof}[Proof of Theorem \ref{thm:BLDLimit}.]
  By the characterization Theorem
  \ref{thm:BLDDefinitions} the classes of 
  $L$-BLD mappings and discrete $L$-LQ mappings equal.
  By Lemma \ref{lemma:LQLimit} $f$
  is $L$-LQ, and thus $L$-BLD.
\end{proof}

\begin{remark}
  Also the ultralimit of a sequence of $L$-LQ 
  mappings is $L$-LQ; see
  \cite[Lemma 3.1]{LeDonnePankka}. Since completeness,
  local compactness and a 
  path-metric
  pass to ultralimits, see e.g.\ \cite{Kapovich-Book},
  results of this section hold also when the convergence
  of mapping packages is replaced by an ultralimit.
  Likewise Theorem \ref{thm:BLDLimit2} has a corresponding 
  ultralimit version. The proof is similar in the ultralimit
  setting as is given here for 
  pointed mapping package convergence.
\end{remark}

The limit of a sequence of $L$-BLD mappings
need not be discrete in general; for example
a sequence of $1$-BLD-mappings $S^1 \to S^1(\frac{1}{k})$,
$z \mapsto \frac{1}{k} z^k$, converges to a constant map
$S^1 \to \{ 0 \}$.
In what follows we consider
the setting of generalized
manifolds of type $A$, where the existence of certain uniform
local multiplicity bounds enable us to 
show that the limit of $L$-BLD mappings is always discrete.

\subsection{Generalized manifolds}

Throughout this section we
assume that $M$ and $N$ are generalized $n$-manifolds
having a complete path-metrics. 
For the definition of generalized manifolds and
their basic theory we refer to \cite{HeinonenRickman}.
Note that generalized manifolds are locally compact.

The following lemma is an elementary observation in the
local value distribution theory of BLD-mappings. We assume
it is known to the specialists in the field but have not
found it recorded in the literature.
\begin{lemma}\label{lemma:BabyValueDistribution}
  Let $f \colon M \to N$ be an $L$-BLD-mapping between
  two generalized manifolds equipped with a complete
  path-metric.
  Then for any two points $x,y \in N \setminus f (B_f)$
  there exists a bijection
  $\psi_f \colon f \inv (\{ x \}) \to f \inv (\{ y \})$
  satisfying $d_M(a,\psi_f(a)) \leq L d_N(x,y)$
  for all $a \in f \inv (\{ x \})$.
\end{lemma}
\begin{proof}
  The claim follows immediately by connecting $x$ and $y$
  with a geodesic $\beta$ and taking the maximal sequence
  of $f$ liftings of $\beta$, see \cite{Rickman-book}. 
  Indeed, to define the bijection $\psi_f$, let 
  $\beta \colon [0,1] \to N$ be a geodesic with 
  $\beta(0) = x$ and $\beta(1) = y$. For each
  $a \in f \inv (\{ x \})$ the path $\beta$ has exactly
  one lift $\alpha_a$ in the maximal sequence of path liftings
  of $\beta$ under $f$ starting from $a$. 
  By Lemma \ref{lemma:BLDPathLifting} 
  maximal lifts of $\beta$ are total, and we define
  $\psi_f(a) = \alpha_a(1) \in f \inv (\{ y \})$.
  Since $y \in N \setminus f (B_f)$ the mapping $\psi_f$
  is injective, so by
  symmetry the mapping $\psi_f$ is bijective. Since $\beta$
  is a geodesic, $d_X(a,\psi_f(a)) \leq L d_Y(x,y)$ by
  (BLD).
\end{proof}

In the setting of generalized manifolds the multiplicity
bounds of Heinonen and Rickman \cite{HeinonenRickman} 
are at our disposal for the proof of Theorem \ref{thm:BLDLimit2}.
Note that the proof of the following theorem shows that 
the limiting map is discrete in a quantitative sense that
$f$ has the same uniform
local multiplicity bounds as the mappings in the sequence
$(f_i)$. This resut generalizes
the authors previous result
\cite[Theorem 1.2]{Luisto-NoteOnLocalToGlobal} and the proof here
is similar.
\begin{proof}[Proof of Theorem \ref{thm:BLDLimit2}]
  By Lemma \ref{lemma:LQLimit}
  we know that the limiting
  map is $L$-LQ. Thus it suffices, by Theorem
  \ref{thm:BLDDefinitions}, to show that $f$ is discrete.
  We show that for each ball $B_{M}(x_0,r_0)$
  we have $N(f(x_0), f , B_{M}(x_0,r_0)) < \infty$.

  By \cite[Theorem 6.8]{HeinonenRickman}, we have for any $j \in \N$,
  any $x \in M$ and any radius $r > 0$
  \begin{align*}\tag{N}
    N(f_j(x), f_j , B_{M_j}(x,r))
    \leq (L c_{M})^n
    \frac{\mathcal{H}^n ( B_{M_j}(x,\lambda r))}{\mathcal{H}^n(B_{N_j}(f_j(x),(\lambda -1)r/L c_{N}))}
  \end{align*}
  for any $\lambda > 1$,
  where $\mathcal{H}^n$ is the Hausdorff 
  $n$-measure and 
  $c_M$ and $c_N$ are the quasi-convexity 
  constants of the spaces $M$ and $N$, respectively.
  We fix $\lambda = 2$ and note that 
  $c_M = c_N = 1$ for path-metric spaces.
  Since we assumed the spaces $M_j$ and $N_j$ to have uniform constants,
  we have for any radius $r$ a constant $N_0(r)$ depending only on $r$ such that
  \begin{align}
    N(f_j(x), f_j , B_{M_j}(x,r)) 
    \leq N_0(r) \label{eq:N}
  \end{align}
  for any $j \in \N$ and $x \in M_j$.

  Let $x_0 \in M$ and fix a radius $r_0 > 0$.
  We show that
  \begin{align*}
    N(f(x_0), f , B_{M}(x_0,r_0)) 
    \leq N_0(r_0) \equalscolon N_0.
  \end{align*}
  Suppose towards contradiction that 
  \begin{align*}
    N(f(x_0), f , B_{M}(x_0,r_0)) 
    \geq N_0 + 1.
  \end{align*}
  Then there exists $y_0 \in f(B_M(x_0,r_0))$
  having at least $N_0 + 1$ preimages
  in $B_M(x_0,r_0)$, that is,
  \begin{align*}
    B_M(x_0,r_0) \cap f \inv (\{ y_0 \})
    \supset \{ a_0, \ldots , a_{N_0} \}
  \end{align*}
  with $a_i$ mutually disjoint.
  Denote
  \begin{align*}
    \delta \colonequals
    \min_{i = 0, \ldots, N_0} \{ d_M(a_i, \{ a_k \}_{k \neq i}), d_M(a_i, \partial B_M(x_0,r_0)) \}
  \end{align*}
  and note that the balls $B_M(a_i, \delta / 2)$ are disjoint.

  Take $R_0 = 2r_0$ and let $(\eps_j^{(R_0)})$, $(g_j^{(R_0)})$ and 
  $(h_j^{(R_0)})$ be the sequences associated to $R_0$
  in (GH-i).
  For each $j \in \N$ and all $i = 0, \ldots , N_0$
  we fix points
  \begin{align*}
    c_i^j
    \in (g_j^{(R_0)})\inv ( B_M(a_i,\eps_j) )
    \subset M_j.
  \end{align*}
  Since the mappings $g_j^{(R_0)}$ are $\eps_j^{(R_0)}$-quasi-isometries,
  the balls $B_{M_j}(c_i^j,\delta/4)$ will be disjoint in $M_j$ 
  for $\eps_j^{(R_0)} < \delta/4$.
  By (GH-ii) we have, for all $i = 0, \ldots , N_0$,
  that $h_j^{(R_0)} (f_j (c_i^j)) \to f(a_i) = y_0$ as
  $i \to \infty$.
  Denote 
  \begin{align*}
    \delta_i \colonequals d_N(h_j^{(R_0)} (f_j (c_i^j)),y_0).
  \end{align*}
  Since
  the mappings $h_j^{(R_0)}$ are $\eps_j^{(R_0)}$-quasi-isometries,
  the triangle inequality yields that
  \begin{align*}
    d_{N_j}(f_j(c_i^j),f_j(c_k^j)) 
    \leq 4 \eps_j^{(R_0)} + \delta_i
  \end{align*}
  for all $i,k \in \{ 0, \ldots, N_0 \}$.
  Thus, for $\eps_j^{(R_0)} < \delta/ (24 L)$,
  there exists a point 
  \begin{align*}
    z_0 \in
    \bigcap_{k = 0}^{N_0} B_{N_j}(f(c_i^j), \delta/(4L))
    \subset \bigcap_{k = 0}^{N_0} fB_{M_j}(c_i^j, \delta/4).
  \end{align*}
  Since 
  the balls $B_{M_j}(c^j_i,\delta/4)$ 
  are disjoint we have
  \begin{align*}
    \# \left(B_{M_j}(x_0,r_0) \cap f_j \inv (\{ z_0 \}) \right)
    \geq N_0+1.
  \end{align*}
  This is a contradiction with \eqref{eq:N}, so 
  $N(f(x_0),f,r_0) \leq N_0$ and $f$ is discrete.
\end{proof}

\begin{remark}
  The multiplicity bound (N) of Heinonen and Rickman
  holds more generally in the setting of
  Ahlfors $Q$-regular generalized manifolds
  equipped with a complete path-metric if
  $\mathcal{H}^Q(B_f) = 0$. Note, however, that this
  is a very strong assumption; indeed 
  the Heinonen-Rickman conjecture
  \cite[Theorem 6.4]{HeinonenRickman} asks whether the branch
  set of a BLD-mapping $f \colon M \to N$ has zero measure
  in any setting where $M$ and $N$ are not quasi-convex
  generalized $n$-manifolds of type A.
  
  Under the assumption $\mathcal{H}^Q(B_f) = 0$,
  the bound (N) follows from 
  standard covering arguments combined with the co-radiality
  condition (R*) and Lemma \ref{lemma:BabyValueDistribution}.
  Thus also the proof of Theorem \ref{thm:BLDLimit2}
  goes through in this setting.
\end{remark}

% \bibliographystyle{alpha}
% \bibliography{RamiBibliography}

\newcommand{\etalchar}[1]{$^{#1}$}
\def\cprime{$'$}\def\cprime{$'$}

\end{document}